\documentclass{amsart}
\usepackage{amsfonts, amsmath, amssymb, amsthm, esint}
\usepackage{graphicx}
\usepackage{todonotes}
\newtheorem{theorem}{Theorem}[section]
\newtheorem{remark}[theorem]{ Remark}

\newtheorem{corollary}[theorem]{Corollary}

\newtheorem{lemma}[theorem]{Lemma}
\newtheorem{clm}[theorem]{Claim}
\newtheorem{definition}[theorem]{Definition}

\newcommand{\V}{\Vert}
\newcommand{\RR} {\mathbb R}

\newcommand{\pa} {\partial}
\newcommand{\Cal} {\mathcal}

\newcommand{\beq} {\begin{equation}}
\newcommand{\eeq} {\end{equation}}

\newcommand{\Vol}{\operatorname{Vol}}
\newcommand{\diam}{\operatorname{diam}}

\begin{document}
\title[Tubular neighbourhoods, cone conditions and inscribed balls]{Some remarks on nodal geometry in the smooth setting}
\author{Bogdan Georgiev and Mayukh Mukherjee}

\address{Max Planck Institute for Mathematics\\ Vivatsgasse 7\\ 53111 Bonn,
\\ Germany}

\email{bogeor@mpim-bonn.mpg.de}
\email{mathmukherjee@gmail.com}
\begin{abstract}

We consider a Laplace eigenfunction $\varphi_\lambda$ on a smooth closed Riemannian manifold, that is, satisfying $-\Delta \varphi_\lambda = \lambda \varphi_\lambda$. We introduce several observations about the geometry of its vanishing (nodal) set and corresponding nodal domains.

First, we give asymptotic upper and lower bounds on the volume of a tubular neighbourhood around the nodal set of $\varphi_\lambda$. This extends previous work of Jakobson and Mangoubi in case $ (M, g) $ is real-analytic. A significant ingredient in our discussion are some recent techniques due to Logunov (cf. \cite{L1}).

Second, we exhibit some remarks related to the asymptotic geometry of nodal domains. In particular, we observe an analogue of a result of Cheng in higher dimensions regarding the interior opening angle of a nodal domain at a singular point. Further, for nodal domains $\Omega_\lambda$ on which $\varphi_\lambda$ satisfies exponentially small $L^\infty$ bounds, we give some quantitative estimates for radii of inscribed balls.
\end{abstract}
\maketitle
%\todo{Check Abstract}
\section{Introduction}
In this note we consider a closed $n$-dimensional Riemannian manifold $M$ with smooth metric $g$, and the Laplacian (or the Laplace-Beltrami operator) $-\Delta$ on $M$\footnote{We use the analyst's sign convention, namely, $-\Delta$ is positive semidefinite.}. For an eigenvalue $\lambda$ of $-\Delta$ and a corresponding eigenfunction $\varphi_\lambda$, we
recall that a nodal domain $ \Omega_\lambda $ is a connected component of the complement of the nodal set $N_{\varphi_\lambda} := \{ x \in M : \varphi_\lambda (x) = 0\}$. We will denote $N_{\varphi_\lambda}$ by $N_\varphi$ via a slight abuse of notation.

Before we formulate our main questions, let us briefly remark on notation: throughout the paper, $|S|$ denotes the volume of the set $S$. The letters $c, C$ etc. are used to denote constants dependent on $(M, g)$ and independent of $\lambda$. The values of such $c, C$ etc. can vary from line to line. Lastly, the expression $X \lesssim Y$ means $X \leq CY$ for some positive constant $C$, with an analogous meaning for $\gtrsim$, and when $X$ and $Y$ are comparable up to constants (i.e., $X \lesssim Y \lesssim X$), we write $X \sim Y$. 

\subsection{Tubular neighbourhoods}
Let $T_{\varphi, \delta} := \{x \in M : \text{dist}(x, N_\varphi) < \delta\}$, which is the $\delta$-tubular neighbourhood of the nodal set $N_\varphi$. We are interested in deriving upper and lower bounds on the volume of $T_{\varphi, \delta}$ in terms of $\lambda$ and $\delta$ in the setting of a smooth manifold. Before investigating the question further, let us begin with a brief overview and motivation.

We first recall the problem of estimating the size of the $ (n-1) $-Hausdorff measure of the nodal set - the question was raised by Yau (Problem \#74, [Yau82]) who conjectured that
\begin{equation}
	C_1 \sqrt{\lambda} \leq \mathcal{H}^{n-1}(N_{\varphi_\lambda}) \leq C_2 \sqrt{\lambda},
\end{equation}
where $ C_1, C_2 $ are constants that depend on $ (M, g) $.

In a celebrated paper (cf. \cite{DF}), Donnelly and Fefferman were able to confirm Yau's conjecture when $ (M, g) $ is real-analytic. Roughly speaking, their techniques relied on analytic continuation and delicate estimates concerning growth of polynomials. 

Later on, in the smooth case, the question of Yau was extensively investigated further: to name a few, we refer to the works by Sogge-Zelditch, Colding-Minicozzi, Mangoubi, Hezari-Sogge, Hezari-Riviere, etc (cf. \cite{SZ1}, \cite{SZ2}, \cite{CM}, \cite{Man4}, \cite{HS}, \cite{HR})  in terms of the lower bound, and Hardt-Simon, Dong, etc (cf. \cite{HaSi}, \cite{D}) in terms of the upper bound; we also refer to \cite{Z} for a far-reaching survey. As an outcome, non-sharp estimates were obtained. The corresponding methods of study were quite broad in nature, varying from local elliptic PDE estimates on balls of size $ \sim 1/\sqrt{\lambda} $ (also referred to as ``wavelength'' balls) to global techniques studying the wave equation.

Recently, in \cite{L1}, \cite{L2}, Logunov made a significant breakthrough which delivered the lower bound in Yau's conjecture for closed smooth manifolds $ (M, g) $ as well as a polynomial upper bound in terms of $ \lambda $. In a nutshell, his methods utilized delicate combinatorial estimates based on doubling numbers of harmonic functions - as pointed out in \cite{L1}-\cite{L2}, some of these techniques were also developed in collaboration with Malinnikova (cf. \cite{LM} as well).

Now, with the perspective of Yau's conjecture, one can ask about further ``stability'' properties of the nodal set - for example, how is the volume of the tubular neighbourhood $ T_{\varphi, \delta} $ of radius $\delta$ around the nodal set behaving in terms of $ \lambda $ and $ \delta $? According to Jakobson-Mangoubi (see Acknowledgments, \cite{JM}), it seems that initially such a question was posed by M. Sodin.

In the real analytic setting, the question about the volume of a tubular neighbourhood $ T_{\varphi, \delta} $ was studied by Jakobson and Mangoubi (cf. \cite{JM}). They were able to obtain the following sharp bounds:

\begin{theorem}[\cite{JM}]\label{JaMa_RA}
	Let $M$ be a real analytic closed Riemannian manifold. Then we have
	\beq\label{eq:JM-bound}
	\sqrt{\lambda}\delta \lesssim |T_{\varphi, \delta}| \lesssim \sqrt{\lambda}\delta,\eeq
	where $\delta \lesssim \frac{1}{\sqrt{\lambda}}$.
\end{theorem}

As remarked by Jakobson and Mangoubi, such bounds describe a certain regularity property of the nodal set - the upper suggests that the nodal set does not have ``too many needles or very narrow branches''; the lower bound hints that the nodal set ``does not curve too much''.

Concerning the proof, Theorem \ref{JaMa_RA} extends the techniques of Donnelly and Fefferman from \cite{DF} by adding an extra parameter $\delta$ to the proofs of \cite{DF}, and verifying that the key arguments still hold.

With that said, it seems natural to ask the question of obtaining similar bounds on the tubular neighbourhood in the smooth case as well - in this setting one can no longer fully exploit the analytic continuation and polynomial approximation techniques in the spirit of Donnelly and Fefferman.

Our first result states that in the smooth setting we have the following:
\begin{theorem}\label{JaMa}
	Let $ (M,g) $ be a smooth closed Riemannian manifold and let $ \epsilon > 0 $ be a given sufficiently small number. Then there exist constants $ r_0 = r_0(M, g) > 0 $ and $ C_1 = C_1(\epsilon, M, g) > 0 $ such that
	\begin{equation}
		\label{low} |T_{\varphi,\delta}| \geq C_1 \lambda^{1/2 - \epsilon}\delta,
	\end{equation}
where $\delta \in (0, \frac{r_0}{\sqrt{\lambda}})$ is arbitrary. 
	
	On the other hand, there exist positive real numbers $\kappa = \kappa(M, g)$ and $ C_2 = C_2(M, g) $, such that
	\begin{equation}\label{up}
		|T_{\varphi, \delta}| \leq C_2 \lambda^{\kappa}\delta,
	\end{equation}
	where again $\delta $ can be any number in the interval $(0, \frac{r_0}{\sqrt{\lambda}})$.
\end{theorem}
Unfortunately, as one sees in the course of the proof of (\ref{low}), the constant $ C_1 $ goes to $0$ as $ \epsilon $ approaches $ 0 $.

Our methods for proving Theorem \ref{JaMa} involve a combination of the tools of ~\cite{DF}-~\cite{JM}, along with some new insights provided by ~\cite{L1}, ~\cite{L2} and ~\cite{LM}. Particularly, in view of the lower bound in Yau's conjecture and of the results in \cite{JM}, it seems that the bound (\ref{low}) is still not optimal. Of course, the upper bound in (\ref{up}) is just polynomial, and, as would be clear from the proof, improvement of the upper bound would be affected by the corresponding improvement of the upper bound in Yau's conjecture.

\subsection{Some remarks on the asymptotic geometry of nodal domains}
As mentioned in the Abstract, in this subsection, we state several remarks regarding some aspects of the asymptotic nodal geometry.
\subsubsection{Interior cone conditions} In dimension $n = 2$, a famous result of Cheng \cite{Ch} says the following (see also ~\cite{St} for a proof using Brownian motion):

\begin{theorem}
	For a compact Riemannian surface $M$, the nodal set $N_\varphi$ satisfies an interior cone condition with opening angle $\alpha \gtrsim \frac{1}{\sqrt{\lambda}}$.
\end{theorem}

Furthermore, in dim $2$, the nodal lines form an equiangular system at a singular point of the nodal set.

Setting $\dim M \geq 3$, we discuss the question whether at the singular points of the nodal set $N_\varphi$, the nodal set can have arbitrarily small opening angles, or even ``cusp''-like situations, or the nodal set has to self-intersect ``sufficiently transversally''.
We observe that in dimensions $n \geq 3$ the nodal sets satisfies an appropriate ``interior cone condition'', and give an estimate on the opening angle of such a cone in terms of the eigenvalue $\lambda$.

Now, in order to properly state or interpret such a result, one needs to define the concept of ``opening angle'' in dimensions $n \geq 3$. We start by defining precisely the notion of tangent directions in our setting.
\begin{definition}
	Let $\Omega_\lambda$ be a nodal domain and $x \in \pa \Omega_\lambda$, which means that $\varphi_\lambda(x) = 0$. Consider a sequence $x_n \in N_\varphi$ such that $x_n \to x$. Let us assume that in normal coordinates around $x$, $x_n = \text{exp }(r_nv_n)$, where $r_n$ are positive real numbers, and $v_n \in S(T_xM)$, the unit sphere in $T_xM$. Then, we define the space of tangent directions at $x$, denoted by $\Cal{S}_xN_\varphi$ as
	\beq
	\Cal{S}_xN_\varphi = \{v \in S(T_xM) : v = \lim v_n, \text{  where  }x_n \in N_\varphi, x_n \to x\}.\eeq%, \text{ where } x_n = \text{exp }(r_nv_n) \to x \} .\eeq
\end{definition}

Observe that there are more well-studied variants of the above definition,
for example, as due to Clarke or Bouligand (for more details, see ~\cite{R}). With that in place, we now give the following definition of ``opening angle''.
\begin{definition}
	We say that the nodal set $N_\varphi$ satisfies an interior cone condition with opening angle $\alpha$ at $x \in N_\varphi$, if any connected component of $S(T_xM) \setminus \Cal{S}_xN_\varphi$ has an inscribed ball of radius $\gtrsim \alpha$.
\end{definition}

Now we have the following:

\begin{theorem}\label{ICC}
	When $\text{dim }M = 3$, the nodal set $N_\varphi$ satisfies an interior cone condition with angle $\gtrsim \frac{1}{\sqrt{\lambda}}$. When $\text{dim }M = 4$, $N_\varphi$ satisfies an interior cone condition with angle $\gtrsim \frac{1}{\lambda^{7/8}}$. Lastly, when $\text{dim }M \geq 5$, $N_\varphi$ satisfies an interior cone condition with angle $\gtrsim \frac{1}{\lambda}$.
\end{theorem}

\subsubsection{Inscribed balls in a nodal domain}
Let us briefly overview some results related to the width of nodal domains.

Given a nodal domain $\Omega_\lambda$, the inradius $\text{inrad}(\Omega_\lambda)$ is the radius of the largest geodesic ball one can fully inscribe in $\Omega_\lambda$. Like the nodal set, it is also another natural object intrinsically tied to the nodal geometry of eigenfunctions, and it is natural to speculate about optimal inradius bounds. From considerations involving domain monotonicity, one can readily see that $\text{inrad}(\Omega_\lambda) \lesssim \frac{1}{\sqrt{\lambda}}$. As regards lower bounds, it was proved in \cite{Man1}, based on ideas in \cite{NPS}, that any closed Riemannian surface satisfies $\text{inrad}(\Omega_\lambda) \gtrsim \frac{1}{\sqrt{\lambda}}$. Moreover, Mangoubi showed (using complex analytic techniques) that such a ball can be centered at any point of maximum of the eigenfunction $\varphi_\lambda$ inside $\Omega_\lambda$. 

In higher dimensions, nodal domains, particularly with regard to their inradius estimates, appear to be sensitive objects. A heuristic explanation is the following (also see ~\cite{H}). In dimensions $n \geq 3$, a curve has zero capacity, which means that there is virtually no difference in the first Dirichlet eigenvalue of a domain $\Omega$ and $\Omega \setminus \Gamma$, where $\Gamma \subset \Omega$ is a reasonably well-behaved curve. But deletion of a curve (or even a single point) can affect the inradius drastically. Still, in dimension $n \geq 3$, Mangoubi was able to show (\cite{Man2}) that every nodal domain $\Omega_\lambda$ satisfies $\text{inrad}(\Omega_\lambda) \gtrsim \frac{1}{\lambda^{\frac{n - 1}{4} + \frac{1}{2n}}}$. His arguments relied on certain ``asymmetry'' results that we briefly discuss below in Subsection \ref{subsubsec:Asymmetry}. Further, in \cite{GM}, we were able to recover the same bounds, but with the additional information that any such ball of radius $\sim \frac{1}{\lambda^{\frac{n - 1}{4} + \frac{1}{2n}}}$ can be centered at a point of maximum of $\varphi_\lambda$ inside $\Omega_\lambda$. This was derived as a Corollary of a quantitative improvement of a Lieb-type result regarding the ``almost inscribedness'' of a wavelength radius ball $B(x_0, \frac{r_0}{\sqrt{\lambda}})$ inside $\Omega_\lambda$ (for a formal statement, see Theorem \ref{alfulins} below). Moreover, as regards the inner radius of nodal domains in dimensions at least $3$, it seems that the lower bound is still not sharp. It is believed that the inner radius should be much more closer to the wavelength scale.

Now, having this discussion in mind, we also recall the following observation: 
\begin{clm} \label{cl:large-ball}
If for a point $x_0 \in M$ we know that $ |\varphi_\lambda (x_0)| \geq  \beta\V \varphi_\lambda\V_{L^\infty(M)}$, where $\beta$ is a constant independent of $\lambda$, then there exists a ball of radius $ \sim 1/\sqrt{\lambda}$ centered at $x_0$ where $\varphi_\lambda$ does not change sign.
\end{clm}
In other words, there exists a fully inscribed ball of wavelength size centered at $x_0$. This claim follows directly from elliptic bounds on the gradient $|\nabla \varphi_\lambda|$.

We address the question to seek quantitative generalizations of this fact under more relaxed lower bounds on $|\varphi_\lambda(x_0)|$. Theorem \ref{newres} below may be seen as one such quantitative generalization. % of this fact. %We prove that exponentially small lower bounds on $\V \varphi_\lambda\V_{L^\infty(\Omega_\lambda)}$ (see (\ref{cond}) below) forces better inradius bounds
%in other words, we give estimates for the inradius of $\Omega_\lambda$ in terms of lower bounds on $\V\varphi_\lambda\V_{L^\infty(\Omega_\lambda)}$. %The estimates turn out to be better than the
Due to Donnelly-Fefferman (~\cite{DF}), on any wavelength radius geodesic ball $B(x, \frac{1}{\sqrt{\lambda}})$ in a closed Riemannian manifold $M$ with smooth metric, we have that $\sup_{B(x, \frac{1}{\sqrt{\lambda}})} |\varphi_\lambda| \gtrsim e^{-C\sqrt{\lambda}}\V \varphi_\lambda\V_{L^\infty(M)}$. Moreover, it is also true that the exponential bounds given by ~\cite{DF} are rarely saturated (one of the rare counterexamples are Gaussian beams of highest weight spherical harmonics), and in most practical examples, much better bounds hold. Motivated by this, we investigate bounds on the size of inscribed balls which are centered at points $x_0$ for which  $|\varphi_\lambda(x_0)|$ is at most ``exponentially'' small.

We have the following observation:

\begin{theorem}\label{newres}
	Let $M$ be a closed Riemannian manifold of dimension $n \geq 3$ with smooth metric. Further, let $x_0 \in \Omega_\lambda$ be such that $\varphi_\lambda (x_0) = \V \varphi\V_{L^\infty(\Omega_\lambda)}$. %and let $x_0 \in \Omega_\lambda$ be a point in the nodal domain $\Omega_\lambda$. 
	Suppose that
	\beq\label{cond}
	\varphi_\lambda(x_0) \geq 2^{-1/\eta}\V \varphi_\lambda\V_{L^\infty(M)}, \eeq
	where $\eta > 0$ is smaller than a fixed constant $\eta_0$ ($\eta$ may be dependent on $\lambda$). Then there exists an inscribed ball $B(x_0, \rho) \subseteq \Omega_\lambda$ of radius
\beq\label{rad}	
	 \rho \gtrsim
	\text{max}\left( \frac{\eta^{\beta (n)}}{\sqrt{\lambda}}, \frac{1}{\lambda^{\alpha (n)}}\right),
	\eeq
	where $\beta (n) = \frac{(n - 1)(n - 2)}{2n}, \alpha (n) = \frac{n - 1}{4} + \frac{1}{2n}$. Furthermore, such a ball can be centered around any such max point $x_0$.
\end{theorem}

In particular, Theorem \ref{newres} implies the following remark  (cf. Claim \ref{cl:large-ball}):
\begin{corollary}
	If for $x_0$ as above, one has that $ |\varphi_\lambda (x_0)| \gtrsim e^{-\lambda^\mu}  \V \varphi_\lambda\V_{L^\infty(M)}$, where $\mu := 2n\nu/((n-1)(n-2)), \nu > 0$ , then there exists a ball of radius $ \sim \frac{1}{\lambda^{1/2 + \nu}}$ centered at $x_0$ where $\varphi_\lambda$ does not change sign.
\end{corollary}

The proof of Theorem \ref{newres} is based on a combination of Mangoubi's result on rapid growth in narrow domains (reproduced below as Theorem \ref{rapgrow}), and Theorem 1.3 of \cite{GM} (reproduced below as Theorem \ref{alfulins}). %IS BOGI THE BOSS OR WHAT? He is. 

\subsection{Acknowledgements} We would especially like to thank Alexander Logunov for reading a draft version of this paper, and for advice regarding  the modification of Proposition 6.4 of \cite{L1} which allowed us to improve the bound in (\ref{low}) from $\lambda^{1/4}\delta$ to $\lambda^{1/2 - \epsilon}\delta$. We also thank Werner Ballmann, Fanghua Lin, Henrik Matthiesen and Stefan Steinerberger for helpful comments. Finally, we gratefully acknowledge the Max Planck Institute for Mathematics, Bonn for providing ideal working conditions.

\section{Auxiliary results about the frequency function and doubling exponents}\label{subsec:Prelim}

%Before proving Theorem \ref{JaMa}, 
We start by recalling some general preliminaries.

Let $B_1$ denote the unit ball in $\RR^n$, and let $\varphi$ satisfy
 \beq\label{ellip}
 L\varphi = 0
 \eeq
 on $B_1$, where $L$ is a second order elliptic operator with smooth coefficients. $L$ is of the form
 \[
 L u = L_1u  - \varepsilon qu,
 \]
 where
 \[
 L_1u = -\pa_i(a^{ij}\pa_j u).\]
 and we make the following assumptions:\newline
(a) $a^{ij}$ is symmetric and satisfies the ellipticity bounds
 \[
 \kappa_1|\xi|^2 \leq a^{ij}\xi_i\xi_j \leq \kappa_2|\xi|^2.
 \]
 (b) $a^{ij}$, $q$ are bounded by $\V a^{ij}\V_{C^1(\overline{B_1})} \leq K, |q| \leq K$, and we assume that $\varepsilon < \varepsilon_0$, with $\varepsilon_0$ small.

Next, we recall and collect a few relevant facts about doubling exponents and different notions of frequency functions - these include scaling and monotonicity results.

For $\varphi$ satisfying (\ref{ellip}) in $B_1$, define for $r < 1$ the following $r$-growth exponent:
\beq\label{grow}
\beta_r(\varphi) = \text{log }\frac{\sup_{B_1}|\varphi|}{\sup_{B_r}|\varphi|},\eeq
%where $B_r$ is the ball of radius $r$ concentric to $B_1$.

A fundamental result of ~\cite{DF} says the following: %for a rescaled eigenfunction $\varphi$, we have
%\beq\label{df}
%\frac{\beta_r(\varphi)}{\text{log}(1/r)} \lesssim \sqrt{\lambda}.
%eeq
\begin{theorem}\label{thm:DF-bounds}
	There exist constants $ C = C(M, g) > 0 $ and $ r_0(M, g) > 0 $ such that for every point $ p $ in $ M $ and every $ 0 < r < r_0 $ the following growth exponent holds:
	\begin{equation}
		\sup_{B(p, r)} |\varphi_\lambda| \leq  \left( \frac{r}{r'}\right)^{C\sqrt{\lambda}} \sup_{B(p, r')} |\varphi_\lambda|, \quad 0 < r' < r.
	\end{equation}
\end{theorem}

In particular, for a rescaled eigenfunction $\varphi$, we have
\beq\label{df}
\frac{\beta_r(\varphi)}{\text{log}(1/r)} \lesssim \sqrt{\lambda}.
\eeq

In this context, we also recall the concept of frequency function (see ~\cite{GL1}, ~\cite{GL2}). For $u$ satisfying $L_1u = 0$ in $B_1$, define for $a \in B_1$, $r \in (0, 1]$ and $B(a, r) \subset B_1$,
\begin{align*}
D(a, r) & = \int_{B(a, r)}|\nabla u|^2dV, \\
H(a, r) & = \int_{\pa B(a, r)} u^2 dS.
\end{align*}

Then, define the generalized frequency of $\varphi$ by
\beq\label{def: freq-func}
\tilde{N}(a, r) = \frac{rD(a, r)}{H(a, r)}.
\eeq

We note that ~\cite{L1} and ~\cite{L2} use a variant of $\tilde{N}(a, r)$, defined as follows:
\beq\label{def: freq-Log}
N(a, r) = \frac{rH^{'}(a, r)}{2H(a, r)}.
\eeq

To pass between $\beta_r(\varphi), \tilde{N}(a, r)$ and $N(a, r)$, we record the following facts: from equation (3.1.22) of ~\cite{HL}, we have that
\beq\label{eq: HL}
H^{'}(a, r) = \left( \frac{n - 1}{r} + O(1)\right) H(a, r)  + 2D(a, r),
\eeq
where $O(1)$ is a function of geodesic polar coordinates $(r, \theta)$ bounded in absolute value by a constant $C$ independent of $r$.
More precisely, in \cite{HL} a certain normalizing factor $\mu$ is introduced in the integrand in the definitions of $H(a,r)$ and $D(a,r)$. As it turns out by the construction, $C_1 \leq \mu \leq C_2$  where $C_1, C_2$ depend on the ellipticity constants of the PDE, the dimension $n$ and a bound on the coefficients (cf. 3.1.11, \cite{HL}).

This gives us that when $\tilde{N}(a, r)$ is large, we have,
\beq \label{13}
N(a, r) \sim \tilde{N}(a, r).
\eeq

Also, it is clear from the proof of Remark 3.1.4 of ~\cite{HL} that
\beq\label{14}
\tilde{N}(a, r) \gtrsim \beta_r(\varphi).
\eeq

We also remark that the frequency $N(a, r)$ is almost-monotonic in the following sense: for any $\varepsilon > 0$, there exists $R > 0$ such that if $r_1 < r_2 < R$, then 
\beq\label{ineq:almost-mon}
N(a, r_1) \leq N(a, r_2)(1 + \varepsilon).
\eeq
This follows from (\ref{eq: HL}) above and standard properties of $\tilde{N}(a, r)$ derived in \cite{HL}.

%Finally, we wish to argue that for small enough balls $B_a(r) \subset M$, $N(a, r) \lesssim \sqrt{\lambda}$. From what has gone above, it suffices to show that $\tilde{N(a, r)}
 
 As regards growth exponents $\beta$, of particular importance to us is the so-called doubling exponent of $\varphi_\lambda$ at a point, which corresponds to the case $r' = \frac{1}{2}r$ in Theorem \ref{thm:DF-bounds}, and is defined as 
 \beq\label{def:doubling-exp}
 \Cal{N}(x, r) = \log \frac{\sup_{B(x, 2r} |\varphi_\lambda|}{\sup_{B(x, r} |\varphi_\lambda|}.
 \eeq

Now, consider an eigenfunction $\varphi_\lambda$ on $M$. We convert $\varphi_\lambda$ into a harmonic function in the following standard way. Let us consider the Riemannian product manifold $\bar{M} :=  M \times \mathbb{R}$ - a cylinder over $M$, equipped with the standard product metric $\bar{g}$. By a direct check, the function
\begin{equation}\label{eq:harm_func}
	u(x, t) := e^{\sqrt{\lambda} t} \varphi_\lambda (x)
\end{equation}
is harmonic.

%We recall the central bound of Donnelly-Fefferman (cf. \cite{DF}):

%\begin{theorem}
	%There exist constants $ C = C(M, g) > 0 $ and $ r_0(M, g) > 0 $ such that for every point $ p $ in $ M $ and every $ 0 < r < r_0 $ the following doubling exponent holds:
	%\begin{equation}
		%\log \frac{\sup_{B(p, 2r)} |\varphi_\lambda|}{\sup_{B(p, r)} |\varphi_\lambda|} \leq C \sqrt{\lambda}.
	%\end{equation}
%\end{theorem}

Hence, by Theorem \ref{thm:DF-bounds}, the harmonic function $ u $ in (\ref{eq:harm_func}) has a doubling exponent which is also bounded by $ C \sqrt{\lambda} $ in balls whose radius is no greater than $ r_1= r_1(M,g) > 0$.

It is well-known that doubling conditions imply upper bounds on the frequency (cf. Lemma $ 6 $, \cite{BL}):

\begin{lemma} \label{lem:Frequency-Bound}
	For each point $p = (x, t) \in \bar{M} $ the harmonic function $ u(p) $ satisfies the following frequency bound:
	\begin{equation}
		\tilde{N}(p, r) \leq C \sqrt{\lambda},
	\end{equation}
	where $ r $ is any number in the interval $ (0, r_2), r_2 = r_2(M, g) $ and $ C > 0 $ is a fixed constant depending only on $ M, g $.
\end{lemma}

For a proof of Lemma \ref{lem:Frequency-Bound} we refer to Lemma $ 6 $, \cite{BL}.

\section{Tubular neighbourhood of nodal set: Theorem \ref{JaMa}}\label{Sec:2}

\subsection{Idea of Proof}
We first focus on the proof of the lower bound. Since the proof is somewhat long and technical, we begin by giving a brief sketch of the overall idea of the proof.

It is well-known by a Harnack inequality argument (see ~\cite{Br} for example), that the nodal set of $ \varphi_\lambda $ is wavelength dense in $M$, which means that one can find $ \sim \lambda^{n/2}$ many disjoint balls $B^i_{\frac{r}{\sqrt{\lambda}}} := B(x_i, \frac{r}{\sqrt{\lambda}}) \subset M$ such that $\varphi_\lambda(x_i) = 0$. Now, to obtain a lower bound on $ |T_{\varphi, \delta}| $ we wish to estimate $ |T_{\varphi, \delta} \cap B(x_i, \frac{r}{\sqrt{\lambda}})|$ from below. The strategy is to consider separately those balls $B^i_{\frac{r}{\sqrt{\lambda}}}$ on which $\varphi_\lambda$ has controlled doubling exponent, which we deal with using the tools of \cite{DF, JM}, and those on which $\varphi_\lambda$ has high doubling exponent, for which we bring in the tools of \cite{L1, L2}. In other words we distinguish two options:

\begin{enumerate}

\item First, for a ball $B(x, \rho)$ of controlled doubling exponent (where $\rho \sim \frac{1}{\sqrt{\lambda}}$, and $\varphi_\lambda (x) = 0$), we show that

\beq\label{ineq:cont-grow}
\frac{|T_{\varphi, \delta} \cap B(x, \rho)|}{\rho^{n - 1}\delta} \geq c.
\eeq

To verify this, we essentially follow the argument of Jakobson and Mangoubi, \cite{JM} from the real-analytic case. The main observation is the fact that the volumes of positivity and negativity of $\varphi_\lambda$ inside such $B(x, \rho)$ are comparable. A further application of the Brunn-Minkowski inequality then yields (\ref{ineq:cont-grow}).

\item Now, to continue the idea of the proof, for a ball $B(x, \rho)$ of high doubling exponent $N$ (where $\rho \sim \frac{1}{\sqrt{\lambda}}$, and $\varphi_\lambda (x) = 0$), we prove that
\beq\label{ineq:uncont-grow}
\frac{|T_{\varphi, \delta} \cap B(x, \rho)|}{\rho^{n - 1}\delta} \geq \frac{1}{N^\varepsilon}, \text{ where } N \lesssim \sqrt{\lambda}.
\eeq

To prove this, we use the following sort of iteration procedure. Using the methods of Logunov, \cite{L1}, \cite{L2}, one first sees that in such a ball $B(x, \rho)$ of large doubling exponent one can find a large collection of smaller disjoint balls $ \{B^j\}_j$, whose centers are again zeros of $\varphi_\lambda$. We then focus on estimating $ |T_{\varphi, \delta} \cap B^j| $ and again distinguish the same two options - either the doubling exponent of $B^j$ is small, which brings us back to the previous case $(1)$ where we have appropriate estimates on the tube, or the doubling exponent of $B^j$ is large. Now, in case the doubling exponent of $B^j$ is large, we similarly discover another large subcollection of even smaller disjoint balls inside $B^j$, whose centers are zeros of $\varphi_\lambda$ and so forth.

We repeat this iteration either until the current small ball has a controlled doubling exponent, or until the current small ball is of radius comparable to the width $\delta$ of the tube $T_{\varphi, \delta} $. In both situations we have a lower estimate on the volume of the tube which brings us to $(\ref{ineq:uncont-grow})$.

\end{enumerate}

Once this is done, (\ref{low}) follows by adding (\ref{ineq:cont-grow}) and (\ref{ineq:uncont-grow}) over $\sim \lambda^{n/2}$ balls $B^i_{\frac{r}{\sqrt{\lambda}}}$, as mentioned above.

\begin{remark} We make a quick digression here and recall that in the real analytic setting, it is known that one can find $\sim \lambda^{n/2}$ many balls of wavelength (comparable) radius, as mentioned above, such that all of them have controlled doubling exponent - in other words, the first case above is the only one that needs to be considered. However, in the smooth setting, it is still a matter of investigation how large a proportion of the wavelength balls possesses controlled doubling exponent. For example, it is shown in ~\cite{CM}, that one can arrange that $\lambda^{\frac{n + 1}{4}}$ such balls possess controlled growth. More explicitly, the following question seems to be of interest and may also have substantial applications in the study of nodal geometry: given a closed smooth manifold $M$, how many disjoint balls $B(x^i_\lambda, \frac{r}{\sqrt{\lambda}})$ of controlled doubling exponent can one find inside $M$ such that $\varphi_\lambda(x^i_\lambda) = 0$, where $r$ is a suitably chosen constant depending only on the geometry of $(M, g)$?.
\end{remark}

The idea of proof of the upper bound (\ref{up}) is quite simple. We take a cube $Q$ inside $M$ of side-length $1$, say, and we chop it up into subcubes $Q_{k}$ of side-length $\delta$. Observe that due to Logunov's resolution of the Nadirashvili conjecture (\cite{L1}), for each subcube $Q_k$ which intersects the nodal set (which we call nodal subcubes following \cite{JM}), we have a local lower bound of the kind $\mathcal{H}^{n - 1}(N_\varphi \cap Q_k) \gtrsim \delta^{n - 1} $. Summing this up, we get an upper bound on the number of nodal subcubes, and in turn, an upper bound on the volume of all nodal subcubes in terms of $\Cal{H}^{n - 1}(N_\varphi)$. %which gives an upper bound on the volume of all nodal subcubes. 
Now, since $T_{\varphi, \delta}$ is contained inside the union of all such nodal subcubes, combined with the upper bound on $\Cal{H}^{n - 1}(N_\varphi)$ due to \cite{L2}, we have (\ref{up}).

\subsection{Proof of Theorem \ref{JaMa}}

\begin{proof}[Proof of (\ref{low})]

We use the notation above and work in the product manifold $ \bar{M} $ with the harmonic function $u(x, t) = e^{\sqrt{\lambda} t} \varphi_\lambda (x)$. For the purpose of the proof of (\ref{low}), we will assume that $M$ is $n - 1$ dimensional. All this is strictly for notational convenience and ease of presentation, as we will now work with the tubular neighbourhood of $ u $, which then becomes $n$ dimensional. Considering the tubular neighbourhood of $u$ instead of $\varphi_\lambda$ does not create any problems because the nodal set of $ u $ is a product, i.e.
\begin{equation}
	\{ u = 0 \} = \{ \varphi_\lambda = 0 \} \times \mathbb{R}.
\end{equation}

As the tubular neighbourhoods we are considering are of at most wavelength radius and at the this scale the Riemannian metric is almost the Euclidean one, we have
\begin{equation}
	T_{u, \frac{\delta}{2}} \subseteq T_{\varphi, \delta} \times \mathbb{R}.
\end{equation}
Hence, to obtain a lower bound for $ |T_{\varphi, \delta}| $ it suffices to bound $ |T_{u, \frac{\delta}{2}}| $ below. To this end, we consider a strip $ S:=M \times [0, R_0] $ where $ R_0 > 0 $ is sufficiently large. %We cover $ S $ by wavelength balls $ \{ B_{\frac{r}{\sqrt{\lambda}}}^i (p_i) \} $ where $ u(p_i) = 0 $ (it is well-known by a Harnack inequality argument, see ~\cite{Br} for example, that the nodal set of $ \varphi_\lambda $ is wavelength dense).

We will obtain lower bound on $ |B_{\frac{r}{\sqrt{\lambda}}}^i (p_i) \cap T_{u, \frac{\delta}{2}}| $, which will give the analogous statements for (\ref{ineq:cont-grow}) and (\ref{ineq:uncont-grow}) for the function $u$. As mentioned before, depending on the doubling exponent of $ u $ in the ball $ B_{\frac{r}{\sqrt{\lambda}}}^i (p_i) $ we distinguish two cases, and we will prove that $\frac{|T_{u,  \frac{\delta}{2}} \cap B(x, \rho)|}{\rho^{n - 1}\delta} \geq c$ in the case of controlled doubling exponent, and $\frac{|T_{u,  \frac{\delta}{2}} \cap B(x, \rho)|}{\rho^{n - 1}\delta} \geq \frac{1}{N^\varepsilon}, \text{ where } N \lesssim \sqrt{\lambda}$ in the case of high doubling exponent.

{\bf	Case I : Controlled doubling exponent:}

	In the regime of controlled doubling exponent, in which case it is classically known that the nodal geometry is well-behaved, we essentially follow the proof in ~\cite{JM}. Let $B := B(p, \rho)$ be a ball such that $u(p) = 0$ and $u$ has bounded doubling exponent on $B(p, \rho)$, that is, $\frac{\sup_{B(p, 2\rho)}|u|}{\sup_{B(p, \rho)}|u|} \leq C$ (ultimately we will set $\rho \sim \frac{1}{\sqrt{\lambda}}$). Then, by symmetry results (see (\ref{man}) below), we have that $C_1 < \frac{|B^+|}{|B^-|} < C_2 $, where $B^+ = \{ u > 0\} \cap B, B^- = \{ u < 0\} \cap B$.

	Let  $\delta := \tilde{c} \rho$, where $\tilde{c}$ is a small constant to be selected later. Denoting by $B^+_\delta$ the $\delta$-neighbourhood of $B^+$, and similarly for $B^-$, and $2B:= B(p, 2\rho)$, we have that since $T_{u, \delta} \supset B^{+}_\delta\cap B^{-}_\delta$,
	
	\begin{equation}
		|T_{u, \delta} \cap 2B| \geq |B^+_\delta| + |B^-_\delta| - |B(p, \rho + \delta)|.
	\end{equation}

By the Brunn-Minkowski inequality, we see that $|B^+_\delta| \geq |B^+| + n\omega_n^{1/n}\delta|B^+|^{1 - 1/n}$, where $\omega_n$ is the volume of the $n$-dimensional unit ball. Setting $|B^+| = \alpha|B|, |B^-| = (1 - \alpha)|B|$, we have

\begin{equation} \label{bm}
	|T_{u, \delta} \cap 2B| \geq \omega_n\left(\rho^n - (\rho + \delta)^n + n\rho^{n - 1}\delta(\alpha^{1 - 1/n} + (1 - \alpha)^{1 - 1/n})\right).
\end{equation}

By asymmetry, $\alpha$ is bounded away from $0$ and $1$, meaning that $\alpha^{1 - 1/n} + (1 - \alpha)^{1 - 1/n} > 1 + C$. Now, taking $\tilde{c}$ small enough, the right hand side of (\ref{bm}) is actually $\gtrsim \rho^{n - 1}\delta$, giving us

\begin{lemma} \label{lem:Step-I}
	Let the tubular distance $ \delta $ and the radius of the ball $ \rho $ be in proportion $ \frac{\delta}{\rho} \leq \tilde{c} $ where $ \tilde{c} > 0$ is a small fixed number. Assume that the doubling index of $ u $ over the ball $ B_\rho $ is small. Then
	\begin{equation} \label{rho}
		|T_{u, \delta} \cap 2B| \gtrsim \rho^{n - 1}\delta.
	\end{equation}
\end{lemma}

%\begin{remark}
	%It is well-known in the smooth setting, as opposed to the real analytic case, that one cannot claim that most of the small wavelength balls have controlled eigenfunction growth. Indeed, it is shown in ~\cite{CM}, that one can arrange that $\lambda^{\frac{n + 1}{4}}$ such balls can have controlled growth, which still seems to be the best estimate known.
%\end{remark}

{\bf Case II: Large doubling exponent:}

	Now, let us consider a ball $B(p, \rho)$ with radius $\rho$ comparable to the wavelength, and let $B^{'} = B(p, \frac{\rho}{2})$. Let us assume that initially we take $ \rho $ such that $ \frac{\delta}{\rho} \leq \tilde{c} $.
	
	Suppose $\frac{\sup_{B^{'}}|u|}{\sup_{\frac{1}{2}B^{'}}|u|}$ is large. By \label{pp7}(\ref{13}) and (\ref{14}), the frequency function $ N(p, \frac{\rho}{2}) $ is also large. Recall also the almost monotonicity of the frequency function $N(x, r)$, given by (\ref{ineq:almost-mon}), which will be implicit in our calculations below. %Referring to Subsection \ref{subsec:Growth} in the Appendix below, we see that
	%By Lemma 3.2 of \cite{L1}, there exists a number $ s \in [\rho, 2\rho) $, such that for any $ t $ in a small interval $ I $ around $ s $ %(cf. Subsection \ref{subsec:Growth})
	%we have
	%\begin{equation} \label{eq:Frequency-Comparison}
	%	N < \tilde{N}(p, t) \leq 2eN.
	%\end{equation}
	
	%If $ N $ is small, we refer back to the previous Case I as the doubling index will be small by (\ref{13}). Hence, we will assume that $ N $ is sufficiently large.
	
	We will make use of the following fact:
	
	\begin{theorem} \label{thm:Number-of-Zeros}
		Consider a harmonic function $u$ on $B(p, 2\rho)$. If $N(p, \rho)$ is sufficiently large, then there is a number $N$ with 
    \begin{equation}\label{eq:Frequency-Comparison}
		N(p, \rho)/10 < N < 2N(p, \frac{3}{2}\rho).
	\end{equation} such that the following holds:
   Suppose that $ \epsilon \in (0, 1) $ is fixed. Then there exists a constant $ C = C(\epsilon) > 0 $ and at least $ [ N^\epsilon]^{n-1} 2^{C \log N / \log \log N} $ disjoint balls $ B(x_i, \frac{\rho}{N^\epsilon  \log^6 N}) \subset B(p, 2\rho) $ such that $ u(x_i) = 0 $. Here $ [.] $ denotes the integer part of a given number.
	\end{theorem}
	
    Theorem \ref{thm:Number-of-Zeros} is mentioned as a remark at the end of Section $6$ of \cite{L1} - for completeness and convenience, we %discuss the details of Theorem \ref{thm:Number-of-Zeros} in the Appendix below.
    give full details of the proof of Theorem \ref{thm:Number-of-Zeros} in this paper, but we relegate them to the Appendix below.
	
	We will now use Theorem \ref{thm:Number-of-Zeros} in an iteration procedure. The first step of the iteration proceeds as follows.
	
	Let us denote by $ \zeta_1 $ the radius of the small balls prescribed by Theorem \ref{thm:Number-of-Zeros}, i.e.
	\begin{equation}
		\zeta_1 := \frac{\rho}{N^\epsilon \log^6 N}.
	\end{equation}
	Further, let $\Cal{B}_1$ denote the collection of these small balls inside $B(p, 2\rho)$. Let $F_1 := \inf_{B \in \Cal{B}_1} \frac{|T_{u, \delta} \cap B|}{\zeta_1^{n-1 }\delta}$ and let us assume that it is attained on the ball $B^1 \in \Cal{B}_1$.

	We then have that
	\begin{align*}
		|T_{u, \delta} \cap B(p, 2\rho)| & \geq \sum_{B_i \in \mathcal{B}_1} |T_{u, \delta} \cap B_i| \geq [N^\epsilon]^{n - 1}2^{C\log N /\log \log N} F_1\zeta_1^{n - 1}\delta \geq \\
		& \geq [N^\epsilon]^{n - 1}2^{C\text{log }N /\log \log N}\frac{\rho^{n - 1} \delta}{(2 N^{\epsilon}\text{log}^6N)^{n - 1}}  F_1,
	\end{align*}
which implies that
\begin{equation}
	\frac{|T_{u, \delta} \cap B(p, 2\rho)|}{\rho^{n - 1}\delta} \geq 2^{C\log N/\log \log N}F_1 \geq F_1,
\end{equation}
by reducing the constant $C$, if necessary, and assuming that $ N $ is large enough. Recalling that by assumption $ F_1 = \frac{|T_{u, \delta} \cap B^1|}{\zeta_1^{n-1}\delta} $, we obtain

\begin{equation}
	\frac{|T_{u, \delta} \cap B(p, 2\rho)|}{\rho^{n - 1}\delta} \geq 2^{C\log N/\log \log N} \frac{|T_{u, \delta} \cap B^1|}{\zeta_1^{n-1}\delta}.
\end{equation}

This concludes the first step of the iteration.

Now, the second step of the iteration process proceeds as follows. We inspect three sub-cases.

\begin{itemize}
	\item First, suppose that $ \delta $ and $ \zeta_1 $ are comparable in the sense that
	\begin{equation}
		\frac{8\delta}{\zeta_1} > \tilde{c},
	\end{equation}
	where $ \tilde{c} $ is the constant from Lemma \ref{lem:Step-I}. As there is a ball of radius $ \delta $ centered at $ x_1 $ (the center of $ B^1 $) that is contained in the tubular neighbourhood, we obtain
	\begin{equation}
		\frac{|T_{u, \delta} \cap B^1|}{\zeta_1^{n-1}\delta} \geq \frac{ C (\tilde{c} \zeta_1)^n }{\zeta_1^{n-1}\delta} \geq C \tilde{c}^n \frac{ \zeta_1}{\delta}.
	\end{equation}
	Furthermore, initially we assumed that $ \frac{\delta}{\rho} \leq \tilde{c} $, hence
	\begin{equation}
		\frac{|T_{u, \delta} \cap B^1|}{\zeta_1^{n-1}\delta} \geq C_1 \tilde{c}^{n-1} \frac{\zeta_1}{\rho} =  C_1 \tilde{c}^{n-1} \frac{1}{N^\epsilon \log^6 N} \geq C_2 \frac{1}{N^{\epsilon_1}},
	\end{equation}
	where $ \epsilon_1 > 0 $ is slightly larger than $ \epsilon $. In combination with the frequency bound of Lemma \ref{lem:Frequency-Bound} and the fact that $ N $ is comparable to the frequency by (\ref{eq:Frequency-Comparison}) we get
	\begin{equation}
		\frac{|T_{u, \delta} \cap B(p, 2\rho)|}{\rho^{n - 1}\delta} \geq  \frac{|T_{u, \delta} \cap B^1|}{\zeta_1^{n-1}\delta} \geq \frac{C_3}{\lambda^{\epsilon_1 /2}}.
	\end{equation}
	
	The iteration process finishes.
	
	\item Now suppose that the tubular radius is quite smaller in comparison to the radius of the ball, i.e.
	\begin{equation}
		\frac{8\delta}{\zeta_1} \leq \tilde{c}.
	\end{equation}
	
	Suppose further that the doubling exponent of $u$ in $ \frac{1}{8} B^1 $ is small. We can revert back to Case I and Lemma \ref{lem:Step-I} by which we deduce that
	\begin{equation}
		\frac{|T_{u, \delta} \cap B(p, 2\rho)|}{\rho^{n - 1}\delta} \geq \frac{|T_{u, \delta} \cap B^1|}{\zeta_1^{n-1}\delta} \geq \frac{|T_{u, \delta} \cap \frac{1}{8} B^1|}{\zeta_1^{n-1}\delta} \geq C,
	\end{equation}
	whence the iteration process stops.
	
	\item Finally, let us suppose that
	\begin{equation}
		\frac{8\delta}{\zeta_1} \leq \tilde{c},
	\end{equation}
	and further that the doubling exponent of $u$ in $ B^1 $ is sufficiently large. We can now replace the initial starting ball $ B(p, 2\rho) $ by $ B^1 $ and then repeat the first step of the iteration process for $ \frac{1}{8} B^1 $. As above, we see that there has to be a ball $ \tilde{B}^1 $ of radius $ \tilde{\zeta}_1 \in (\frac{1}{4}\zeta_1, \frac{1}{2} \zeta_1) $ upon which the frequency is comparable to a sufficiently large number $ N_1 $. %Moreover,
%	\begin{equation}
%		\frac{|T_{u, \delta} \cap B_\rho|}{\rho^{n - 1}\delta} \geq 2^{C\log N/\log \log N} \frac{|T_{u, \delta} \cap B^1|}{\zeta_1^{n-1}\delta} \geq \frac{|T_{u, \delta} \cap \tilde{B}^1|}{\tilde{\zeta}_1^{n-1}\delta} =: \tilde{F}_1,
%	\end{equation}
%	provided $ N $ is large enough.
	
	Now, we apply Theorem \ref{thm:Number-of-Zeros} and within $ B^1 $ discover at least $ [N_1^\epsilon]^{n-1} 2^{C \log N_1 / \log \log N_1} $ balls of radius
	\begin{equation}
		\zeta_2 := \frac{\zeta_1}{N_1^\epsilon \log^6 N_1},
	\end{equation}
    such that $\varphi_\lambda$ vanishes at the center of these balls.
	
	As before, we denote the collection of these balls by $ \mathcal{B}_2 $ and put $ F_2 := \inf_{B \in \Cal{B}_2} \frac{|T_{u, \delta} \cap B|}{\zeta_1^{n-1 }\delta}$. Analogously we also obtain
	\begin{equation}
		F_1 = \frac{|T_{u, \delta} \cap B^1|}{\zeta_1^{n - 1}\delta} \geq 2^{C\log N_1/\log \log N_1}F_2 \geq F_2.
	\end{equation}
	
	Again, we reach the three sub-cases. If either of the two first sub-cases holds, then we bound $ F_2 $ in the same way as $ F_1 $ - this yields a bound on $ \frac{|T_{u, \delta} \cap B_\rho|}{\rho^{n - 1}\delta} $. If the third sub-case holds, then we repeat the construction and eventually get $ F_3, F_4, \dots $.
	
\end{itemize}

Notice that the iteration procedure eventually stops. Indeed, it can only proceed if the third sub-case is constantly iterated. However, at each iteration the radius of the considered balls drops sufficiently fast and this ensures that either of the first two sub-cases is eventually reached.

This finally gives us
\begin{equation}\label{ineq:final_iteration}
	\frac{|T_{u, \delta} \cap B(p, 2\rho)|}{\rho^{n - 1}\delta} \geq F_1 \geq F_2 \geq \dots \geq \frac{C_3}{\lambda^{\epsilon_1 /2}}.
\end{equation}

At last, we are done with the iteration, and this also brings us to the end of the discussion about Case I and Case II. To summarize what we have established, the most ``unfavourable'' situation is that scenario in Case II, where we at every level of the iteration we encounter balls of high doubling exponent, and we have to carry out the iteration all the way till the radius of the smaller balls (whose existence at every stage is guaranteed by Theorem \ref{thm:Number-of-Zeros}) drops below $\delta$. The lower bound for $\frac{|T_{u, \delta} \cap B(p, 2\rho)|}{\rho^{n - 1}\delta}$ in such a ``worst'' scenario is given by (\ref{ineq:final_iteration}).

We are now ready to finish the proof. Letting $\rho = \frac{r}{2\sqrt{\lambda}}$ and by summing (\ref{ineq:final_iteration}) over the $\sim \lambda^{n/2}$ many wavelength balls $B^i_{\frac{r}{\sqrt{\lambda}}}$ (as mentioned at the beginning of this Section), we have that
\begin{equation}
	|T_{u, \delta}| \geq \frac{C_3}{\lambda^{\epsilon_1 /2}}\rho^{n - 1}\delta\lambda^{n/2} \gtrsim \lambda^{1/2 - \epsilon_1/2}\delta.
\end{equation}

Using the relationship between the nodal sets of $\varphi_\lambda$ and $u$, this yields (\ref{low}).

\end{proof}

Now we turn to the proof of the upper bound.

\begin{proof}[Proof of (\ref{up})]

We start by giving a formal statement of the main result of ~\cite{L2}:
\begin{theorem}\label{th: log-upp-bd}
	Let $(M, g)$ be a compact smooth Riemannian manifold without boundary. Then there exists a number $\kappa$, depending only on $n = \text{dim } M$ and $C = C(M, g)$ such that
	\begin{equation}
		\Cal{H}^{n - 1}(N_\varphi) \leq C\lambda^\kappa.
	\end{equation}
\end{theorem}

As remarked before, we assume that $M$ has sufficiently large injectivity radius. Consider a finite covering $Q_k$ of $M$ by cubes of side length $1$, say. Consider a subdivision of each cube $Q_k$ into subcubes $Q_{k, \nu}$ of side length $\delta,$ where $\delta \leq \frac{1}{3}$. Call a small subcube $Q_{k, \nu}$ a nodal cube if $N_\varphi \cap Q_{k, \nu} \neq \emptyset$. Also, denote by $Q_{k, \nu}^*$ the union of $Q_{k, \nu}$ with its $3^n - 1$ neighbouring subcubes. Then, it is clear that
\beq\label{Nod}
T_{\varphi, \delta} \subset \bigcup_{\text{Nod}} Q_{k, \nu}^*,
\eeq
where Nod denotes the set of all nodal subcubes $Q_{k,
	\nu}$.

%We know that at least half the subcubes $Q_{k, \nu}$ have bounded eigenfunction growth, or are ``good subcubes''. Using this, from an application of the Brunn-Minkowski inequality, as in the proof of (\ref{low}), we see that%\todo{Check again: "as in the proof of"}
%\[
%\liminf_{t \to 0}\frac{|T_{\varphi, t}|}{t} \gtrsim \delta^{n - 1}.
%\]

%We know that $N_\varphi$ is rectifiable (see, for example, ~\cite{B}), which means that the above limit exists and is given by $\Cal{H}^{n - 1}(N_\varphi \cap Q^*_{k,\nu})$, which means that
By Theorem 1.2 of \cite{L1}, we have that
\beq\label{hell}
\Cal{H}^{n - 1}(N_\varphi \cap Q^*_{k,\nu}) \gtrsim \delta^{n - 1}.
\eeq

Summing up (\ref{hell}), %and using the fact that the nodal set is controlled in a good cube,
we get that
\begin{align*}
3^n\Cal{H}^{n - 1}(N_\varphi) %& \geq \sum_{k, \nu}\Cal{H}^{n - 1}(N_\varphi \cap Q^*_{k,\nu}) \\
& \geq \sum_{\text{Nod} } \Cal{H}^{n - 1}(N_\varphi \cap Q^*_{k,\nu}) \\& \gtrsim \#(\text{nodal } Q_{k,\nu})\delta^{n - 1},
\end{align*}
which means that the number of nodal subcubes is $\lesssim \Cal{H}^{n - 1}(N_\varphi)/\delta^{n - 1}$. %This gives us that the number of bad nodal subcubes is $\lesssim \Cal{H}^{n - 1}(N_\varphi)/\delta^{n - 1}$. Taken together and
Using (\ref{Nod}), this means that
\[
|T_{\varphi, \delta}| \lesssim \Cal{H}^{n - 1}(N_\varphi)\delta.
\]

Finally, we invoke Theorem \ref{th: log-upp-bd} to finish our proof.
S%Note that here $\delta$ needs to be on the same scale as $B$, which means that
%$\delta \sim 1/\lambda^{\frac{1}{2\gamma}} = 1/\lambda^\mu$, where $\mu$ depends only on $M$.

\end{proof}

\section{Some remarks on the asymptotic geometry of nodal domains: Theorems \ref{ICC} and \ref{newres}}\label{ICCsec}
\subsection{Internal cone condition}
\subsubsection{Preliminaries}\label{prel1}
We will use Bers scaling of eigenfunctions near zeros (see ~\cite{Be}). We quote the version as appears in ~\cite{Z}, Section 3.11.
\begin{theorem}[Bers]\label{Zildo}
Assume that $\varphi_\lambda$ vanishes to order $k$ at $x_0$. Let $\varphi_\lambda(x) = \varphi_k(x) + \varphi_{k + 1}(x) + ..... $ denote the Taylor expansion of $\varphi_\lambda$ into homogeneous terms in normal coordinates $x$ centered at $x_0$. Then $\varphi_\kappa(x)$ is a Euclidean harmonic homogeneous polynomial of degree $k$.
\end{theorem}

We also use the following inradius estimate for real analytic metrics (see ~\cite{G}).
\begin{theorem}\label{Zeldiev}
	Let $(M, g)$ be a real-analytic closed manifold of dimension at
	least $3$. If $\Omega_\lambda$ is a nodal domain corresponding to the eigenfunction $\varphi_\lambda$, %where $\lambda \geq \lambda_0$ (fixed constant),
	then there exist constants $\lambda_0, c_1$ and $c_2$ which depend only
	on $(M, g)$, such that
	\beq\label{ra_inrad}
	\frac{c_1}{\lambda} \leq \text{inrad }(\Omega_\lambda) \leq  \frac{c_2}{\sqrt{\lambda}}, \lambda \geq \lambda_0.\eeq
\end{theorem}

Since the statement of Theorem \ref{Zeldiev} is asymptotic in nature, we need to justify that if $\lambda < \lambda_0$, a nodal domain corresponding to $\lambda$ will still satisfy $\text{inrad }(\Omega_\lambda) \geq \frac{c_3}{\lambda}$ for some constant $c_3$. This follows from the inradius estimates of Mangoubi in \cite{Man2}, which hold for all frequencies. Consequently, we can assume that every nodal domain $\Omega$ on $S^n$ corresponding to the spherical harmonic $\varphi_k(x)$, as in Theorem \ref{Zildo}, has inradius $\gtrsim \frac{1}{\lambda}$.
\subsection{Proof of Theorem \ref{ICC}}
	We observe that Theorem \ref{Zeldiev} applies to spherical harmonics, and in particular the function $\text{exp}^*(\varphi_k)$, restricted to $S(T_{x_0}M)$, where $\varphi_k(x)$ is the homogeneous harmonic polynomial given by Theorem \ref{Zildo}. %, which is the geodesic sphere of radius $1$ around $x_0$.
	 Also, a nodal domain for any spherical harmonic on $S^2$ (respectively, $S^3$) corresponding to eigenvalue $\lambda$ has inradius $\sim \frac{1}{\sqrt{\lambda}}$ (respectively, $\gtrsim \frac{1}{\lambda^{7/8}}$).
	
	With that in place, it suffices to prove that
	\beq\label{suff}
	\Cal{S}_{x_0}N_\varphi \subseteq \Cal{S}_{x_0}N_{\varphi_k}.
	\eeq
\begin{proof} By definition, $v \in \Cal{S}_{x_0}N_\varphi$ if there exists a sequence $x_n \in N_\varphi$ such that $x_n \to x_0$, $x_n = \text{exp} (r_nv_n)$, where $r_n$ are positive real numbers and $v_n \in S(T_{x_0}M)$, and $v_n \to v$.

 This gives us,
\begin{align*}
0 & = \varphi_\lambda(x_n)  = \varphi_\lambda(r_n\text{exp }v_n) \\
& = r_n^k\varphi_k(\text{exp }v_n) + \sum_{m > k}r_n^{m}\varphi_m(\text{exp }v_n)\\
& = \varphi_k(\text{exp }v_n) + \sum_{m > k}r_n^{m - k}\varphi_m(\text{exp }v_n)\\
& \to \varphi_k(\text{exp }v), \text{ as } n \to \infty.
\end{align*}

 Observing that $\varphi_k(x)$ is homogeneous, this proves (\ref{suff}).
\end{proof}

\subsection{Inscribed balls in a nodal domain}
\subsubsection{Preliminaries} \label{prel}

We start again by collecting some auxiliary results that we need for the proof of Theorem  \ref{newres}. These include
\begin{enumerate}
\item A maximum principle for solutions of elliptic PDE,
\item Comparison estimates on the volumes of positivity/negativity of eigenfunctions (i.e. local asymmetry of sign),
\item Growth of solutions of elliptic PDE in narrow domains,
\item Existence of almost inscribed balls.
\end{enumerate}

\subsubsection{Local elliptic maximum principle}  We quote the following local maximum principle, which appears as Theorem 9.20 in ~\cite{GT}.
\begin{theorem}
	Suppose $Lu \leq 0$ on $B_1$. Then
	\beq\label{GT}
	\sup_{B(y, r_1)} u \leq C(r_1/r_2, p)\left( \frac{1}{\Vol (B(y, r_2))}\int_{B(y, r_2)} (u^{+}(x))^p dx\right)^{1/p},
	\eeq
	for all $p > 0$, whenever $0 < r_1 < r_2$ and $B(y, r_2) \subseteq B_1$.
	
\end{theorem}

\subsubsection{Local asymmetry of nodal domains} \label{subsubsec:Asymmetry} Our proof also uses the concept of local asymmetry of nodal domains, which roughly means the following. Consider a manifold $M$ with smooth metric. If the nodal set of an eigenfunction $\varphi_\lambda$ enters sufficiently deeply into a geodesic ball $B$, then the volume ratio between the positivity and negativity set of $\varphi_\lambda$ in $B$ is controlled in terms of $\lambda$. More formally, we have the following result from ~\cite{Man2}:

\beq\label{man}
	\frac{|\{\varphi_\lambda > 0\} \cap B|}{|B|} \gtrsim \frac{1}{\langle \beta_{1/2}(\varphi )\rangle^{n - 1}},
\eeq
where $\langle \beta_r\rangle = \text{max}\{\beta_r, 3\}$. In particular, when combined with the growth bound of Donnelly-Fefferman, this yields that 
$$
\frac{|\{\varphi_\lambda > 0\} \cap B|}{|B|} \gtrsim \frac{1}{\lambda^{\frac{n - 1}{2}}}.
$$
This particular question about comparing the volumes of positivity and negativity seems to originate from \cite{ChMu}, \cite{DF1}, and then work of Nazarov, Polterovich and Sodin (cf. \cite{NPS}), where they also conjecture that the present bound is far from being optimal. Moreover, it is believed that the sets of positivity and negativity should have volumes which are comparable up to a factor of $1/\lambda^\epsilon$ for small $\epsilon > 0$.

\subsubsection{Rapid growth in narrow domains} Heuristically, this means that if $\varphi$ solves (\ref{ellip}), and has a deep and narrow positivity component, then $\varphi$ grows rapidly in the said component. In our paper, we use an iterated version of this principle, which appears as Theorem 3.2 in ~\cite{Man3}. Let $\varphi_\lambda$ satisfy (\ref{ellip}) on the rescaled ball $B_1$, as at the beginning of Section \ref{subsec:Prelim}.
\begin{theorem}\label{rapgrow}
	Let $0 < r'< 1/2$. Let $\Omega$ be a connected component of $\{ \varphi > 0\}$ which intersects $B_{r'}$. Let $\eta > 0$ be small. If $|\Omega \cap B_r|/|B_r| \leq \eta^{n - 1}$ for all $r' < r < 1$, then
	\[
	\frac{\sup_\Omega \varphi}{\sup_{\Omega \cap B_{r'}}\varphi} \geq \left(\frac{1}{r'}\right)^{C^{'}/\eta},
	\]
	where $C^{'}$ is a constant depending only on the metric $(M, g)$.
\end{theorem}
\subsubsection{Almost inscribing wavelength ball} 
We finally recall some results discussing ``almost'' inscribed balls inside a given domain. More precisely, we start by recalling a celebrated theorem of Lieb (see \cite{L}), which considers the case of a domain $\Omega \subset \RR^n$ and states that there exists a point $x \in \Omega$, and a ball $B:= B(x,\frac{r}{\sqrt{\lambda_1(\Omega)}})$ of radius $\frac{r}{\sqrt{\lambda_1(\Omega)}}, $ (here $r> 0$ is sufficiently small) which is ``almost'' inscribed in $\Omega$, that is

\begin{equation}
	\frac{|B \cap \Omega |}{|B|} \geq 1 - \epsilon.
\end{equation}

Here $\lambda_1(\Omega)$ is the first Dirichlet eigenvalue of $\Omega$. Moreover, $\epsilon$ approaches $0$ as $r \rightarrow 0$.

A further related result was obtained in the paper \cite{MS} (see, in particular, Theorem 1.1 and Subsection 5.1 of \cite{MS}).

In ~\cite{GM}, a refinement of the above statement of Lieb was obtained stating that $x \in \Omega$ can be taken as any point where the first Dirichlet eigenfunction of $\Omega$ (assumed to be positive without loss of generality) reaches a maximum. 

Specifying these statements to nodal domains, we have:
\begin{theorem}\label{alfulins}
	Let $\dim M \geq 3, \epsilon_0 > 0 $ be fixed and $ x_0 \in \Omega_\lambda $ be such that $ |\varphi_\lambda(x_0)| = \text{max}_{\Omega_\lambda}|\varphi_\lambda| $. There exists $ r_0 = r_0 (\epsilon_0) $, such that
	\begin{equation}\label{Vol}
	\frac{|B_{r_0} \cap \Omega_\lambda |}{|B_{r_0}|} \geq 1 - \epsilon_0,
	\end{equation}
	where $ B_{r_0} $ denotes $ B\left(x_0, \frac{r_0}{\sqrt{\lambda}}\right) $.
\end{theorem}

It is also important for our discussion below to have a relation between  $\epsilon_0$ and $r_0$. Referring to Corollary 1.4 in ~\cite{GM}, we have that they are related by
\beq\label{rel}
r_0 = C \epsilon_0^{\frac{n - 2}{2n}},
\eeq
where $C$ is a constant depending only on $(M, g)$.

\subsubsection{Idea of proof of Theorem \ref{newres}} Before going into the details of the proof, let us first outline the main ideas. Let us define $ B := B(x_0, \frac{r_0}{\sqrt{\lambda}})$ where $x_0$ is a point of maximum as stated in Theorem \ref{newres} and $r_0 > 0$ is a sufficiently small number. Also recall that $\varphi_\lambda(x_0)$ is assumed to be bounded below in terms of $\eta$.

Now, roughly speaking, we will see that if $r_0$ is sufficiently small in terms of $\eta$, then $\varphi_\lambda$ does not vanish in $\frac{1}{4} B$, a concentric ball of quarter radius. This will imply the claim of the Theorem.

To this end, we argue by contradiction (i.e. we assume that $\varphi_\lambda$ vanishes in $\frac{1}{4} B$) and follow the three steps below:

\begin{enumerate}
\item First, Theorem \ref{alfulins} above tells us that we can ``almost'' inscribe a ball $B = B(x_0,\frac{r_0}{\sqrt{\lambda}})$ inside $\Omega_\lambda$, up to the error of certain ``spikes'' of total volume $\epsilon_0 |B|$ entering the ball, where $\epsilon_0$ and $\rho$ are related by (\ref{rel}). In particular, if we assume w.l.o.g. that $\varphi_\lambda$ is positive on $\Omega_\lambda$, then the volume $|\{ \varphi_\lambda < 0 \} \cap B|$ is relatively small and does not exceed $\epsilon_0 |B|$.

\item The second step consists in showing that the sup norms of $\varphi^{-}$ and $\varphi^{+}$ %egative and positive values of $\varphi_\lambda$ 
in the spikes are comparable. More formally, observe that on each connected component of $\frac{1}{4}B \setminus \Omega_\lambda$ (i.e., on each spike in $\frac{1}{4}B$), $\varphi_\lambda$ can be positive or negative a priori. However, by a relatively simple argument involving the mean value property of harmonic functions and standard elliptic maximum principles, we show that on $\frac{1}{4}B \setminus \Omega_\lambda$, $\sup \varphi_\lambda^{-} \lesssim \sup\varphi_\lambda^{+}$.  %proving that the doubling exponent of $\varphi_\lambda$ in $\frac{1}{8}B$ is bounded, by proving that $\varphi_\lam.

\item %Bounded growth implies symmetry of sign and conclude no zero points.
For the third step, we begin by noting that if we can show that the doubling exponent of $\varphi_\lambda$ in $\frac{1}{8}B$ is bounded above by a constant, then the asymmetry estimate (\ref{man}) will give us that the set $\{\varphi_\lambda < 0\} \cap \frac{1}{4}B$ has a large volume, which contradicts Step ($1$) above. This will be a contradiction to our assumption that $\varphi_\lambda$ vanishes somewhere in $\frac{1}{4} B$, and thus we finally conclude that $\frac{1}{4}B$ is fully inscribed inside $\Omega_\lambda$.

Now, the bounded doubling exponent will be ensured, if $\varphi_\lambda(x_0)$ controls (up to a constant) all the values of $\varphi_\lambda$ inside $\frac{1}{4}B$. Using the input from Step ($2$) above as well as the a priori assumption on $\varphi_\lambda(x_0)$, it suffices to ensure that $\varphi^{+}$ is suitably bounded. This is where we bring in the rapid growth in narrow domains result (Theorem \ref{rapgrow}).
\end{enumerate}

\subsubsection{Proof of Theorem \ref{newres}}\label{main}

%Now we give a proof of Theorem \ref{newres}.
\begin{proof}

{\bf Step $1$: An almost inscribed ball:}

As before, let $x_0$ denote the max point of $\varphi_\lambda$ in the nodal domain $\Omega_\lambda$. Let us assume the sup-norm bound (\ref{cond}) and let us set $B := B(x_0, \frac{r_0}{\sqrt{\lambda}})$ be a ball centered at $x_0$ and of radius $\frac{r_0}{\sqrt{\lambda}}$, where $r_0 > 0$ is sufficiently small and determined below. Further, let us denote the non-inscribed ``error-set'' by $X := B \setminus \Omega_\lambda$.

We start by making the following choice of parameters: we select $0 < \epsilon_0 \leq (\eta C^{'})^{n - 1} $ with a corresponding $r_0 := C \epsilon_0^{\frac{n-2}{2n}}$ (prescribed by (\ref{rel})), where $C^{'}$ is the constant coming from Theorem \ref{rapgrow}; moreover we assume that $0 < \eta \leq \eta_0$ for some fixed small positive number $\eta_0$, so that by Theorem \ref{alfulins}, the relative volume of the ``error'' set $X$ is sufficiently small, i.e. 
\beq \label{err}
\frac{|X \cap \frac{1}{4}B|}{|\frac{1}{4}B|} \lesssim \frac{4^n|X \cap B|}{|B|} \leq 4^n \epsilon_0 =: C_2,
\eeq
where $C_2 > 0$ is appropriately chosen below.

We claim that in fact $\varphi_\lambda$ does not vanish in $\frac{1}{4} B $, the concentric ball of a quarter radius.

To prove this, we will argue by contradiction - that is, let us suppose that $\varphi_\lambda$ vanishes somewhere in $\frac{1}{4} B$.

{\bf Step $2$: Comparability of $\varphi_\lambda^+$ and $\varphi_\lambda^-$:}
By assuming the contrary, let $x$ be a point in $\overline{X} \cap \frac{1}{4}B$ lying on the boundary of a spike, that is, $\varphi_\lambda (x) = 0$.
Consider a ball $B^{'}$ around $x$ with radius $\frac{r_0}{2\sqrt{\lambda}}$.
Since $\varphi_\lambda (x) = 0$, we have that (up to constants depending on $(M, g)$),
\beq
\frac{1}{|B^{'}|}\int_{B^{'}} \varphi_\lambda^{-} \sim \frac{1}{|B^{'}|}\int_{B^{'}} \varphi_\lambda^{+}.
\eeq
%where $B^{''}$ is a ball slightly smaller than and fully contained in $B^{'}$.
This follows from mean value properties of harmonic functions; for a detailed proof, see Lemma 5 of ~\cite{CM}.

Now, let $B^{''}$ be a ball slightly smaller than and fully contained in $B^{'}$. Using the local maximum principle (\ref{GT}), we have that (up to constants depending on $(M, g)$),
\beq
\sup_{B^{''}} \varphi_\lambda^{-} \lesssim \frac{1}{|B^{'}|} \int_{B^{'}} \varphi_\lambda^{-} \lesssim \frac{1}{|B^{'}|}\int_{B^{'}} \varphi_\lambda^{+} \leq \sup_{B^{'}} \varphi_\lambda^{+}.
\eeq
This shows that in order to bound $\varphi_\lambda^{-}$, it suffices to bound $\varphi_\lambda^{+}$.
This finishes Step ($2$). 

{\bf Step $3$: Controlled doubling exponent and conclusion:}

Our aim is to be able to bound $\sup_{\frac{1}{4} B}\varphi_\lambda^{+}$ in terms of $\varphi_\lambda (x_0)$, as that would give us control of the doubling exponent of $\varphi_\lambda$ on $\frac{1}{8}B$. In other words, we wish to %investigate for what kind of $C(\lambda)$, we have that
establish that
\beq\label{want}
\sup_{\frac{1}{4}B} \varphi_\lambda^{+} \leq C\varphi_\lambda (x_0),
\eeq
where $C$ is a constant independent of $\lambda$.

If $X \cap \frac{1}{2}B \cap \{\varphi > 0\} = \emptyset$, %$X \cap B^{'} \cap \{\varphi > 0\} = \emptyset$, 
then (\ref{want}) follows immediately by definition. % with $C(\lambda) = \text{constant}$.
Otherwise, calling $X^{'} := X \cap \frac{1}{2}B$, %$X^{'} := X \cap B^{'}$, 
let $\Omega^{'}_\lambda$ represent another nodal domain on which $\varphi_\lambda$ is positive and which intersects $X^{'}$. In other words, %$\Omega^{'}_\lambda \cap B^{'}$ 
$\Omega^{'}_\lambda \cap \frac{1}{2}B$ gives us a spike entering %$B^{'}$ 
$\frac{1}{2}B$ which $\varphi_\lambda$ is positive, and our aim is to obtain bounds on this spike. %satisfies $\varphi_\lambda|_{\Omega^{'}_\lambda} > 0$, and $\Omega^{'}_\lambda \cap X^{'} \neq 0$.

%Now we invoke Theorem \ref{rapgrow}, which is 

Observe that (\ref{err}) implies that the volume of the spike %$\Omega^{'}_\lambda \cap B^{'}$ 
$\Omega^{'}_\lambda \cap \frac{1}{2}B$ is small compared to %$|B^{'}|$, 
$\frac{1}{2}B$, and this allows us to invoke Theorem \ref{rapgrow}. %and our assumption that $\epsilon_0 = (\eta C^{'})^{n - 1}$. %By (\ref{cond}) and Theorem \ref{rapgrow}, w
We see that
\begin{align*}
    2^{1/\eta}\varphi_\lambda(x_0) & \gtrsim \V\varphi_\lambda\V_{L^\infty(M)} \text{  (by hypothesis (\ref{cond})})\\
    & \geq \sup_{\Omega^{'}_\lambda} \varphi_\lambda \geq   2^{1/\eta}\sup_{\Omega^{'}_\lambda \cap \frac{1}{2}B} \varphi_\lambda 
    \geq 2^{1/\eta}\sup_{\Omega^{'}_\lambda \cap \frac{1}{4}B} \varphi_\lambda \text{  (by applying Theorem \ref{rapgrow}}).
\end{align*}

%From the condition (\ref{cond}), using $\eta = \psi$,
%Observe that to invoke Theorem \ref{rapgrow}, we have to use
Now (\ref{want}) follows, which implies that the growth is controlled in the ball $\frac{1}{8}B$, %B(x_0, \frac{r_0}{8\sqrt{\lambda}})$, 
that is,
\beq\label{lab}
\beta_{1/8}(\varphi_\lambda) = \frac{\sup_{\frac{1}{4}B}|\varphi_\lambda|}{\sup_{\frac{1}{8}B}|\varphi_\lambda|} \leq c_1,
\eeq
where $c_1$ depends on $(M, g)$ 
and not on $\epsilon_0$ or $\lambda$ (in particular, not on $r_0, \eta$).

Now, we bring in the asymmetry estimate (\ref{man}), which, together with (\ref{lab}), tells us that
\beq\label{35}
\frac{|\{\varphi_\lambda < 0\} \cap \frac{1}{4}B|}{|\frac{1}{4}B|} \geq c_2,
\eeq
where $c_2$ is a constant depending only on $c_1$ and $(M, g)$. But selecting the constant $C_2$ to be smaller than $c_2$ we see that (\ref{35}) contradicts (\ref{err}). Hence, we obtain a contradiction with the fact that the function $\varphi_\lambda$ vanishes inside $\frac{1}{4} B$.

Finally, this proves that with the initial choice of parameters, there is an inscribed ball of radius $ \frac{r_0}{4 \sqrt{\lambda}}$ inside $\Omega_\lambda$. By construction, we had that $r_0 \sim \eta^{\frac{(n - 1)(n - 2)}{2n}} = \eta^{\beta(n)}$.

Combined with the inner radius estimates in \cite{Man2}, this proves the claim of Theorem \ref{newres}.

\end{proof}

\section{Appendix: number of zeros over balls with large doubling exponent}

We address the proof of Theorem \ref{thm:Number-of-Zeros}. We will essentially just follow and adjust Section 6, \cite{L1} - there Theorem \ref{thm:Number-of-Zeros} was stated as a remark. For mere completeness, we will recall all the relevant statements.

Let us briefly give an overview of how the proof proceeds.

First, we consider a harmonic function in a ball and gather a few estimates on the way $ u $ grows near a point of maximum. The discussion here involves classical harmonic function estimates as well as scaling of the frequency function $ N(p, r) $ (cf. Subsection \ref{subsec:Prelim}) and the doubling numbers.

Second, let us consider a cube $ Q $ and divide it into small equal subcubes. We recall a combinatorial result (Theorem 5.2, \cite{L1}) which, roughly speaking, gives quantitative estimate on the number of small bad subcubes (i.e., subcubes with large doubling exponent) of a given cube $ Q $.

Third, we utilize the results in the first two steps to prove Theorem \ref{thm:Number-of-Zeros}.

A few words regarding notation: given a point $O \in M$, we take a small enough coordinate chart $(U, \psi)$ around $O$ such that the Riemannian metric $g$ on the chart is comparable to the Euclidean metric in the following sense: given $\nu > 0$, there is a sufficiently small $R_0 = R_0 (\nu, M, g, O)$ such that $(1 - \nu)d_g(x, y) < d_{\text{Euc}}(\psi(x),\psi (y)) < (1 + \nu)d_g(x, y)$ for any two distinct points $x, y \in B_g(O, R_0)$. Under this metric comparability, we will drop the subscript ``$g$'' henceforth, and will describe ``cubes'' and ``boxes'' and their partitions, and such combinatorial ideas directly on the manifold $M$.

\subsection{Growth of harmonic functions near a point of maximum}\label{subsec:Growth}

Let us start by recalling the following observation (Lemma 3.2, \cite{L1}). Let $ B(p, 2r) \subset B(O, R_0) $ where the frequency function satisfies $ N(p, \frac{r}{2}) > 10 $. Then there exists numbers $ s \in [r, \frac{3}{2}) $ and $ N \geq 5 $ so that
\begin{equation}
	N \leq N(p, t) \leq 2e N,
\end{equation}
where the parameter $ t $ is any number within the interval $ I $ given by
\begin{equation}
	I := \left( s(1 - \frac{1}{1000 \log^2 N}), s(1 + \frac{1}{1000 \log^2 N}) \right).
\end{equation}

In words, we find and work in a small spherical layer where the frequency is comparable to $ N $.

Recalling the function $ H(p, t) = \int_{\partial B(p,t)} u^2 dS$, it follows from the definition of the frequency function that
\begin{equation}
	\frac{H(x, r_2)}{H(x, r_1)} = \exp \left( 2 \int_{r_1}^{r_2} \frac{N(x, r)}{r} dr  \right).
\end{equation}
Combining this with the control over $ N $ in the interval $ I $, we obtain
\begin{equation}
	\left( \frac{t_2}{t_1} \right)^{2N} \leq \frac{H(p, t_2)}{H(p, t_1)} \leq \left( \frac{t_2}{t_1} \right)^{4eN},
\end{equation}
where $ t_1 < t_2 $ and $ t_1, t_2 \in I $.

Now, let us consider a point of maximum $ x \in \partial B(p, s) $, such that
\begin{equation}
	\sup_{y \in \bar{B}(p, s)} |u(y)| = |u(x)| =: K.
\end{equation}

We now look at concentric spheres of radii $ s^- := s(1-\delta) $ and $ s^+ := s(1 + \delta) $ where $ \delta $ is a small number in the interval $ [\frac{1}{10^6 \log^{100} N}, \frac{1}{10^6 \log^{2} N}] $. We can estimate $ \sup_{B(p, s^+)} |u| $ and $ \sup_{B(p, s^-)} |u| $ in terms of $ K $:

\begin{lemma} [Lemma 4.1, \cite{L1}]
	There exist $ c, C > 0 $ depending on $ M, g, n, O, R_0 $, such that
	\begin{equation} \label{eq:9}
		\sup_{B(p, s^-)} |u| \leq CK 2^{-c \delta N},
	\end{equation}
	\begin{equation}
		\sup_{B(p, s^+)} |u| \leq CK 2^{C \delta N}.
	\end{equation}
\end{lemma}

\begin{proof}[Sketch of Proof]
	The proof uses the above scaling for $ H(p, t) $ and classical estimates for harmonic functions. For a detailed discussion we refer to \cite{L1}.
\end{proof}

Let us recall the classical doubling number $ \mathcal{N}(x, r) $ (cf. Subsection \ref{subsec:Prelim}), which was defined as
\begin{equation}
	2^{\mathcal{N}(x, r)} = \frac{\sup_{B(x, 2r)} |u|}{\sup_{B(x, r)} |u|}.
\end{equation}

Let us recall the following result (cf. Appendix, \cite{L2}):

\begin{lemma}
	Let $ \epsilon > 0 $ be fixed. There exist numbers $ R_0 > 0, C > 0 $ such that for $ r_1, r_2 $ with $ 2r_1 \leq r_2 $ and $ B(x, r_2) \subset B(O, R_0) $, we have the following estimate
	\begin{equation}
		\left( \frac{r_2}{r_1} \right)^{\mathcal{N}(x, r_1)(1-\epsilon) - C} \leq \frac{\sup_{B(x, r_2)} |u|}{\sup_{B(x, r_1)} |u|} \leq \left( \frac{r_2}{r_1} \right)^{\mathcal{N}(x, r_2)(1+\epsilon) + C}.
	\end{equation}
	
	In particular,
	\begin{equation}
		\mathcal{N}(x, r_1)(1-\epsilon) - C \leq \mathcal{N}(x, r_2)(1+\epsilon) + C.
	\end{equation}
\end{lemma}

As a straightforward corollary of the above discussion we obtain

\begin{lemma}
	There is a constant $ C = C(M, g, n) > 0 $ such that
	\begin{equation}
		\sup_{B(x, \delta s)} |u| \leq K 2^{C\delta N + C}.
	\end{equation}
	Moreover, for any $ \tilde{x} $ with $ d(x, \tilde{x}) \leq \frac{\delta s}{4} $, we have
	\begin{equation} \label{eq:14}
		\mathcal{N}(\tilde{x}, \frac{\delta s}{4}) \leq C \delta N + C,
	\end{equation}
	\begin{equation} \label{eq:15}
		\sup_{B(\tilde{x}, \frac{\delta s}{10 N})} |u| \geq K 2^{-C \delta N \log N - C}.
	\end{equation}
\end{lemma}

For a proof we refer to Lemma 4.2, \cite{L1}.

\subsection{An estimate on the number of bad cubes}

Let $ Q $ be a  given cube. We define the doubling index $ N(Q) $ of the cube $ Q $ by
\begin{equation}
	N(Q) := \sup_{x \in Q, r \leq \diam(Q)} \log \frac{\sup_{B(x, 10nr)} |u|}{\sup_{B(x, r)} |u|}.
\end{equation}

Clearly, $ N(Q) $ is monotonic in the sense that if a cube $ Q_1 $ is contained in the cube $ Q_2 $, then $ N(Q_1) \leq N(Q_2) $. Furthermore, if a cube $ Q $ is covered by a collection of cubes $ \{ Q_i \} $ with $ \diam(Q_i) \geq \diam(Q) $, then there exists a cube $ Q_i $ from the collection, such that $ N(Q_i) \geq N(Q)$.

The main result in this subsection is

\begin{theorem}[Theorem 5.3, \cite{L1}] \label{thm:5.2}
	There exist constants $ c_1, c_2, C > 0 $ and a positive integer $ B_0 $, depending only on the dimension $ n $, and positive numbers $ N_0 = N_0(M, g, n, O), R = R(M, g, n, O) $ such that for any cube $ Q \subset B(O, R) $ the following holds:
	
	If we partition $ Q $ into $ B^n $ equal subcubes, where $ B > B_0 $, then the number of subcubes with doubling exponent greater than $ \max(N(Q) 2^{-c_1 \log B / \log \log B}, N_0 ) $ is less than $ C B^{n-1-c_2} $.
\end{theorem}

The last theorem uses and refines a previous result (Theorem 5.1, \cite{L2}) where roughly speaking the dynamic relation between the size of the small cubes and their doubling index is not estimated with that precision. The discussion proceeds through an iteration argument.

\subsection{Proof of Theorem \ref{thm:Number-of-Zeros}}

	\textbf{Step 1 - the set-up}. We consider the same setting as in Subsection \ref{subsec:Growth}: we have a ball $ B(p, 2r) \subset B(O, R_0) $, numbers $ s \in [r, \frac{3}{2}r ], N \geq 5 $, such that
	\begin{equation}
	N \leq N(p, t) \leq 2eN,
	\end{equation}
	for any $ t \in I $ where $ I $ is the interval defined above.
	
	We also consider a point of maximum $ x \in \partial B(p, s), \sup_{\partial B(p,s)} |u| = |u(x)| =: K $ and a point $ \tilde{x} \in \partial B(p, s(1-\delta))$, such that $ d(x, \tilde{x}) = \delta s $. Here we have introduced the small number $ \delta := \frac{1}{10^8 n^2 \log^2 N} $ (we follow the notation in \cite{L1}, but to avoid confusion, we note that the $\delta$ chosen here is much smaller compared to the $\delta$ used in Subsection \ref{subsec:Growth}). By construction, we have that $ d(x, \tilde{x}) \sim \frac{r}{\log^2 N}$ up to constants depending only on dimension.
	
	Let us denote by $ T $ a (rectangular) box, such that $ x $ and $ \tilde{x} $ are centers of the opposite faces of $ T $ - one side of $ T $ is $ d(x, \tilde{x}) $ and the other $ n-1 $ sides are equal to $ \frac{d(x, \tilde{x})}{[\log N]^4} $, where $ [.] $ denotes the integer part of a given number.
	
	Now, let $ \epsilon \in (0, 1) $ be given. By cutting along the long side of $ T $, we subdivide $ T $ into equal subboxes (referred to as ``tunnels'') $ T_i, i = 1, \dots, [N^\epsilon]^{n-1} $, so that each $ T_i $ has one side of length $ d(x, \tilde{x}) $ and the other $ n-1 $ sides of length $ \frac{d(x, \tilde{x})}{[N^\epsilon][\log N]^4} $.
	
	Further, by cutting perpendicularly to the long side, we divide $ T_i $ into equal cubes $ q_{i, t}, t = 1, \dots, [N^\epsilon][\log N]^4 $ all of which have side-length of $ \frac{d(x, \tilde{x})}{[N^\epsilon][\log N]^4} $ and whose centers are denoted by $ x_{i, t} $. We also arrange the parameter $ t $ so that $ d(q_{i,t}, x) \geq d(q_{i, t+1}, x) $.
	
	We will assume that $ N $ is sufficiently large, i.e. bounded below by $ N_0(n, M, g) > 0 $.
	
	\textbf{Step 2 - growth along a tunnel}. We wish to relate how large $ u $ is at the first and last cubes - $ q_{i, 1} $ and $ q_{i, [N^\epsilon] [\log N]^4} $. To this end we will use the lemmata from Subsection \ref{subsec:Growth}.
	
	First, let us observe that $ q_{i, 1} \subset B (p, s(1 - \frac{\delta}{4})) $. Indeed, for sufficiently large $ N $ we have
	\begin{equation}
		d(p, q_{i, 1}) \leq d(p, \tilde{x}) + d(\tilde{x}, q_{i, 1}) \leq s(1 - \delta) + \frac{C \delta s \sqrt{n}}{[\log N]^4} \leq s \left( 1 -  \frac{\delta}{2} \right).
	\end{equation}
	
	The estimate (\ref{eq:9}) yields
	\begin{equation}
		\sup_{q_{i, 1}} |u| \leq \sup_{B (p, s(1 - \frac{\delta}{4}))} |u| \leq K 2^{-c_1 \frac{N}{\log^2 N} + C_1}.
	\end{equation}
	
	On the other hand, let us denote the last index along the tunnel by $ \tau $, i.e. $ \tau := [N^\epsilon][\log N]^4 $. As the cube $ q_{i, \tau} $ is of size comparable to $ \frac{1}{[N^\epsilon] [\log^6 N]} $ and $ N $ is assumed to be large enough, we can find an inscribed geodesic ball $ B_{i, \tau} \subset \frac{1}{2} q_{i, \tau} $, centered at $ x_{i, \tau} $ and of radius $ \frac{s}{N} $.
	
	Now, by definition $ d(x_{i, \tau}, x) \leq \frac{Cs}{[\log N]^6} $. Hence, the inequality (\ref{eq:15}) implies (taking $ \tilde{x} $ there to be $ x_{i, \tau} $)
	\begin{equation}
		\sup_{q_{i, \tau}} |u| \geq  \sup_{B_{i, \tau}} |u| \geq K 2^{-C_3 \frac{N}{\log^5 N} -C_3}.
	\end{equation}
	
	Putting the last two estimates together, we obtain
	
	\begin{lemma} \label{lem:6.1}
		There exist positive constants $ c, C $ such that
		\begin{equation}
			\sup_{\frac{1}{2} q_{i,[N^\epsilon][\log N]^4 }} |u| \geq \sup_{\frac{1}{2} q_{i, 1}} |u| 2^{cN / \log^2 N - C}.
		\end{equation}
	\end{lemma}

	\textbf{Step 3 - bound on the number of good tunnels}. Next, we show that there are sufficiently many tunnels, such that the doubling exponents of the contained cubes are controlled (cf. Claim 6.2, \cite{L1}). More precisely,

	\begin{lemma}\label{Lem:good-tunn}
		There exist constants $ c = c(\epsilon) > 0, N_0 > 0 $ such that at least half of the tunnels $ T_i $ are ``good'' in the sense that they have the following property:
		
		For each cube $ q_{i, t} \in T_i, t \in 1, \dots, [N^\epsilon][\log N]^4 $ we have
		\begin{equation}
			N(q_{i, t}) \leq \max \left( \frac{N}{2^{c \log N / \log \log N}}, N_0 \right).
		\end{equation}
	\end{lemma}
	
	\begin{proof}
		We assume that $ N $ is sufficiently big. We focus on the cubes that fail to satisfy this condition, i.e. we consider the ``bad'' cubes $ q_{i, t} $ for which
		\begin{equation}
			N(q_{i,t}) > N 2^{-c \log N / \log \log N}.
		\end{equation}
		The constant $ c = c(\epsilon) $ stems from Theorem \ref{thm:5.2} and is addressed below. As the number of all tunnels is $ [N^\epsilon]^{n-1} $, by the pigeonhole principle, the claim of the lemma will follow if one shows that the number of bad cubes does not exceed $ \frac{1}{2} [N^\epsilon]^{n-1} $.
		
		To this end, we apply Theorem \ref{thm:5.2} in the following way. We divide $ T $ into equal Euclidean cubes $ Q_t, t = 1, \dots, [\log N]^4 $ of side-length $ \frac{d(x, \tilde{x})}{[\log N]^4} $. We need to control $ N(Q_t) $ via $ N $. To do this, observe that
		\begin{equation}
			d(x, y) \leq 4 d(x, \tilde{x}) \leq \frac{s}{10^7 \log^2 N},
		\end{equation}
		that is $ y $ is not far from the maximum point. Hence, we can apply (\ref{eq:14}) and obtain
		\begin{equation}
			\frac{\sup_{B(y, \frac{s}{10^7 \log^2 N})} |u|}{\sup_{B(y, \frac{1}{2} \frac{s}{10^7 \log^2 N})} |u|} \leq 2^{C \frac{N}{\log^2 N} + C}.
		\end{equation}
		
		The definition and monotonicity of $ N(Q_t) $ as well as the assumption that $ N > N_0 $ imply that
		\begin{equation}
			N(Q_t) \leq N, \quad t = 1, \dots, [\log N]^4.
		\end{equation}
		
		Now, the application of Theorem \ref{thm:5.2} with $ B = [N^\epsilon] $ gives that the number of bad cubes contained in $ Q_t $ (i.e., cubes whose doubling exponent is greater than $
		\max \left( N(Q_t) 2^{-c_1 \log (N^\epsilon) / \log \log (N^\epsilon)}, N_0 \right) $)  is less than $ C [N^\epsilon]^{n-1-c_2} $. Note that we can absorb the $ \epsilon $ term in the constant $ c_1 $ and deduce that the bad cubes have a doubling exponent greater than $ \max \left( N(Q_t) 2^{-c(\epsilon) \log N / \log \log N}, N_0 \right) $.
		
		Summing over all cubes $ Q_t $ we obtain that the number of all bad cubes in $ T $ is no more than
		
		\begin{equation}
			C[N^\epsilon]^{n-1-c_2} [\log N]^4 \leq \frac{1}{2} [N^\epsilon]^{n-1}.
		\end{equation}
	\end{proof}
	
	\textbf{Step 4 - zeros along a good tunnel}. Finally, we will count zeros of $ u $ along a good tunnel. Roughly, the harmonic function $ u $ has tame growth along a good tunnel. If $ u $ does not change sign, one could use the Harnack inequality to bound the growth of $ u $ in a suitable way. Summing up the growth over all cubes along a tunnel and using the estimate in Step $ 2 $ we obtain (cf. Claim 6.3, \cite{L1}):
	
	\begin{lemma}
		There exists a constant $ c_2 = c_2(\epsilon) > 0 $ such that if $ N $ is sufficiently large and $ T_i $ is a good tunnel, then there are at least $ 2^{c_2 \log N / \log \log N} $ closed cubes $ \bar{q}_{i,t} $ that contain a zero of $ u $.
	\end{lemma}
	
	\begin{proof}
		As the tunnel is good, Lemma \ref{Lem:good-tunn} gives that for every $ t = 1, \dots, [N^\epsilon][\log N]^4 - 1 $ we have
		\begin{equation}
			\log \frac{\sup_{\frac{1}{2} q_{i, t+1}} |u|}{\sup_{\frac{1}{2} q_{i, t}} |u|} \leq \log \frac{\sup_{4 q_{i, t}} |u|}{\sup_{\frac{1}{2} q_{i, t}} |u|} \leq \frac{N}{2^{c_1 \log N / \log \log N}}.
		\end{equation}
		
		We split the index set $ \{ 1, \dots, [N^\epsilon][\log N]^4 - 1 \} $ into two disjoint subsets $ S_1, S_2 $: an index $ t $ is in $ S_1 $ provided $ u $ does not change sign in $ \bar{q}_{i, t} \cup \bar{q}_{i, t+1}  $. The advantage in $ S_1 $ is that one can use the Harnack inequality. For $ t \in S_1 $ we have
		\begin{equation}
			\log \frac{\sup_{\frac{1}{2} q_{i, t+1}} |u|}{\sup_{\frac{1}{2} q_{i, t}} |u|} \leq C_1.
		\end{equation}
		
		Using Lemma \ref{lem:6.1} and summing-up we obtain
		
		\begin{align}
			c\frac{N}{\log^2 N} - C & \leq \log \frac{\sup_{\frac{1}{2} q_{i, [N^\epsilon][\log N]^4}} |u|}{\sup_{\frac{1}{2} q_{i, 1}} |u|} = \sum_{S_1} \log \frac{\sup_{\frac{1}{2} q_{i, t+1}} |u|}{\sup_{\frac{1}{2} q_{i, t}} |u|} + \sum_{S_2} \log \frac{\sup_{\frac{1}{2} q_{i, t+1}} |u|}{\sup_{\frac{1}{2} q_{i, t}} |u|} \leq \\
			& \leq C_1 |S_1| + \frac{N}{2^{c_1 \log N / \log \log N}} |S_2| \leq C_1 [N^\epsilon] \log^4 N + \frac{N}{2^{c_1 \log N / \log \log N}} |S_2| \leq \\
			& \leq \frac{c}{2} \frac{N}{\log^2 N} - C + \frac{N}{2^{c_1 \log N / \log \log N}} |S_2|.
		\end{align}
		
		This shows that
		\begin{equation}
			|S_2| \geq 2^{\frac{c_1}{2} \log N / \log \log N}.
		\end{equation}
	\end{proof}
	
	We have already seen that there are at least $ \frac{1}{2} [N^\epsilon]^{n-1} $ good tunnels, which, by summing-up, means that the number of small cubes, where $ u $ changes sign is at least $ \frac{1}{2} [N^\epsilon]^{n-1}  2^{c_2 \log N / \log \log N}  $.

	Finally, in each cube $ \bar{q}_{i,t} $ let us fix a zero $ x_{i, t} \in \bar{q}_{i,t}, u(x_{i, t}) = 0 $ and note that
	\begin{equation}
		\diam(\bar{q}_{i,t}) \sim \frac{r}{N^\epsilon \log^6 N}.
	\end{equation}
	Each ball $ B(x_{i,t}, \frac{r}{N^\epsilon \log^6 N}) $ intersects at most $ \kappa = \kappa(n) $ other balls of this kind. By taking a maximal disjoint collection of such balls and reducing the constant $ c_2 $ to $ c_3 = c_3(\epsilon) $ we conclude the proof of Theorem \ref{thm:Number-of-Zeros}.

\end{document}